\documentclass[12pt]{amsart}
\usepackage{amssymb}
\usepackage{amsmath}
\usepackage{amsfonts}
\usepackage{times}      
\usepackage[usenames]{color}
\usepackage{amssymb,amsthm,amscd}
\usepackage[all]{xy}
\usepackage{enumerate}

\usepackage{graphicx}
\usepackage{mathbbol}
\usepackage{mathtools}
\usepackage{hyperref}
\usepackage[margin=1in,footskip=0.25in]{geometry}
\usepackage{xcolor}
 \numberwithin{equation}{section}
\newtheorem{theorem}{Theorem}[section]

\newtheorem{corollary}[theorem]{Corollary}   
\newtheorem{lemma}[theorem]{Lemma}   

\newtheorem{definition}[theorem]{Definition}
\newtheorem{remark}[theorem]{Remark}

\newtheorem*{quest*}{Question}
\newtheorem*{thma}{Theorem A}
\newtheorem*{thmb}{Theorem B}

\newcommand{\be}{\begin{equation}}
\newcommand{\ee}{\end{equation}}
\newcommand{\bee}{\begin{equation*}}
\newcommand{\eee}{\end{equation*}}

\newcommand{\R}{{\mathbb{R}}}
\newcommand{\Z}{{\mathbb{Z}}}

\newcommand{\N}{{\mathbb{N}}}
\newcommand{\One}{{\bf \rm{1}}}

\newcommand{\dist}{\operatorname{dist}}
\newcommand{\diam}{\operatorname{diam}}

\newcommand{\vol}{\operatorname{Vol}}

\newcommand{\set}{\rm{set}}
\newcommand{\Lip}{\operatorname{Lip}}

\newcommand{\mass}{{\mathbf M}}
\newcommand{\curr}{{\mathbf M}}         
\newcommand{\intrectcurr}{{\mathcal I}} 
\newcommand{\intcurr}{{\mathbf I}}      

\newcommand{\rstr}{\:\mbox{\rule{0.1ex}{1.2ex}\rule{1.1ex}{0.1ex}}\:}
\newcommand{\bdry}{\partial}

\def\bdm{\begin{displaymath}}
\def\edm{\end{displaymath}}

\theoremstyle{definition}

\usepackage{amsmath,amsopn}

\newcommand{\func}[1]{\operatorname{#1}}

\newtheorem*{acknowledge}{Acknowledgments}

\begin{document}

\title{Alexandrov Spaces with Integral Current Structure}
\author{M. Jaramillo}
\address[Jaramillo]{Department of Mathematics, University of Connecticut, }
\email{maree.jaramillo@uconn.edu}
\urladdr{}
\author{R. Perales}
\address[Perales]{Conacyt Research Fellow. Instituto de Matem\'aticas, Universidad Nacional Aut\'onoma de M\'exico\\
Oaxaca, M\'exico.}
\email{raquel.perales@matem.unam.mx}
\urladdr{}
\author{P. Rajan}
\address[Rajan]{Department of Mathematics, University of California at Riverside}
\email{rajan@math.ucr.edu}
\urladdr{https://sites.google.com/site/priyankarrajangeometry/home}
\author{C. Searle}
\address[Searle]{Department of Mathematics, Statistics, and Physics\\
 Wichita State University, Wichita, Kansas}
\email{searle@math.wichita.edu}
\urladdr{https://sites.google.com/site/catherinesearle1/home}
\author{A. Siffert}
\address[Siffert]{Max Planck Institute for Mathematics, Vivatsgasse 7, 53111 Bonn, Germany}
\email{siffert@mpim-bonn.mpg.de}
\urladdr{}

\begin{abstract} 
We endow each closed, orientable  Alexandrov space $(X,d)$  with an integral current $T$ of weight equal to 1, $\partial T = 0$ and $\set(T)=X$, in other words, we prove that
$(X,d,T)$ is an integral current space with no boundary.
Combining this result with a result of Li and Perales, we show that non-collapsing sequences of these spaces with uniform lower curvature and diameter bounds admit subsequences whose Gromov-Hausdorff and intrinsic flat  limits  agree.
\end{abstract}
\maketitle

\section{Introduction}

There exist a wealth of notions for convergence of sequences of Riemannian manifolds, among them are $C^{k,\alpha}$-convergence (see e.g. \cite{Pe}) and Gromov-Hausdorff convergence of metric spaces \cite{G}. 
More recently, Sormani and Wenger \cite{SW} introduced another notion of convergence, the {\em intrinsic flat convergence} of integral current spaces. 
This notion is based on the flat distance between integral currents in Euclidean space developed by Federer and Fleming in \cite{FF}, which was subsequently extended to metric spaces by Ambrosio and Kirchheim \cite{AK}.  In \cite{SW} Sormani and Wenger motivate the introduction of intrinsic flat convergence by the following example:
spheres with many splines that contain increasingly small amounts of volume converge in the intrinsic flat sense to spheres but not in the Gromov-Hausdorff sense.
There are numerous interesting applications of intrinsic flat convergence, several of them to General Relativity (see e.g. \cite{HLS}, \cite{J}, \cite{LS1}, \cite{LS2}).

\smallskip

The study of Alexandrov spaces  has been largely motivated by the fact that they are
 a generalization of Riemannian manifolds with curvature bounded from below.  
 The relationship between Riemannian and Alexandrov geometry has been used
repeatedly to solve difficult problems in Riemannian geometry: for example,
 Perelman's solution to Thurston's Geometrization
conjecture for $3$-manifolds, where  
a structure
theorem for $3$-manifolds that collapse with a uniform lower curvature bound
played a crucial role 
(see \cite{P1}, \cite{P2}, \cite{P3}). 

Examples of Alexandrov spaces are limits of Gromov-Hausdorff sequences of Riemannian manifolds with sectional curvature bounded from below,
as well as all quotients of Riemannian manifolds with a lower curvature bound under isometric group actions.  
It is a longstanding conjecture that not all Alexandrov spaces belong to the closure of the space of Riemannian manifolds with a lower curvature bound (see e.g. \cite{K2}, \cite{PWZ}).
However, since smooth oriented Riemannian manifolds of finite volume can be seen as integral current spaces, it is then natural to ask:
$$\text{\textit{Which Alexandrov spaces can be endowed with an integral current structure?}}$$

In this paper we begin the study of this question by considering Alexandrov spaces without boundary. We will prove the following theorem:
\begin{thma}\label{t:1}
Let $(X,d)$ be a closed, oriented, $n$-dimensional Alexandrov space with curvature bounded below by $\kappa$.  There exists an integral current structure $T$ with weight equal to 1 defined on $X$ such that
$(X, d, T)$ is an $n$-dimensional integral current space.
\end{thma}

\smallskip

 Li and Perales  \cite{LP} studied sequences of integral current spaces $(X_j,d_j,T_j)$ such that $\bdry T_j=0$. In the noncollapsing case they proved that if the integral current  $T_j$ has weight $1$ and $(X_j,d_j)$ are Alexandrov spaces with uniform lower curvature and upper diameter  bounds, then,  the Gromov-Hausdorff and 
intrinsic flat limits of the sequence $(X_j,d_j,T_j)$ agree.

\begin{theorem} \cite{LP} \label{LP}
Let $(X_j, d_j, T_j)$ be $n$-dimensional integral current spaces with weight $1$ and no boundary. 
Suppose that $(X_j, d_j)$ are Alexandrov spaces with nonnegative curvature and $\mbox{diam}(X)\leq D$.
Then either the sequence converges to the zero integral current space in the intrinsic flat sense
$$(X_j,d_j,T_j)\overset{\frak{F}}\longrightarrow \bold{0}$$
or a subsequence converges in the Gromov-Hausdorff sense and intrinsic flat sense to the same space
$$( X_{j_k},d_{j_k})\overset{\text{GH}}\longrightarrow(X,d)$$
and
$$(X_{j_k},d_{j_k},T_{j_k})\overset{\frak{F}}\longrightarrow(X,d,T).$$
\end{theorem}

Combining Theorem A with Theorem \ref{LP}, the following theorem is then immediate.

\begin{thmb}\label{t:2} 
Let $X_i$ be closed, oriented $n$-dimensional Alexandrov spaces with curvature bounded below by $\kappa$. 
 Suppose further that $\diam(X_i)\leq D$ and the sequence is non-collapsing. 
 Then the $X_i$ can be made into $n$-dimensional integral current spaces such that a subsequence converges in the intrinsic flat and Gromov-Hausdorff sense 
 to the same metric space.
\end{thmb}

 We note that the relationship between intrinsic flat limits and Gromov-Hausdorff limits of sequences has been studied previously. 
  Sormani and Wenger \cite{SW10} proved that the limits agree for sequences of closed Riemannian manifolds with nonnegative Ricci curvature, diameter bounded above, and volume bounded below by a positive constant. For the case of Riemannian manifolds with boundary, Perales  \cite{Per} has shown that under the same conditions the limits agree.
  Munn  \cite{M} proved for closed Riemannian manifolds that this statement is also true when nonnegative Ricci curvature is substituted by two sided bounds on the Ricci curvature. 
  Matveev and Portegies then showed that this result extends to sequences of manifolds with an arbitrary uniform lower bound on Ricci curvature, and that the limiting current is essentially unique \cite{MP}.
  In particular, this tells us that those Alexandrov spaces arising as limits of sequences of Riemannian manifolds with a lower curvature bound admit an integral current structure.

\smallskip
\subsection*{Organization}
The paper is structured as follows. In Section \ref{sec-background}, we introduce the tools and results necessary to the proof of  Theorem A. In Section \ref{3}, we introduce 
a new definition of orientability in Alexandrov spaces using strainers, which will allows us to define an oriented atlas for our spaces in the sense of Federer (see Definition \ref{Federer}). Finally, in Section \ref{sec-t:1}, we prove Theorem A.

\begin{acknowledge} 
The authors are grateful to Christina Sormani for suggesting this problem to us and for stimulating conversations. P. Rajan and C. Searle would also like to thank Frederick Wilhelm for 
numerous helpful conversations. We would also like to 
thank the Banff International Research Station for its support during the Women in Geometry Workshop (15w5135), where the work on this paper was begun. M. Jaramillo was supported by a U Connecticut OVPR/AAUP travel grant.  R. Perales and P. Rajan were supported by Christina Sormani's  NSF-DMS grant \#1309360. P. Rajan was supported in part by Simons Foundation Grant \#358086 (Frederick Wilhelm).  C. Searle  was supported in part by Simons Foundation Grant \#355508 (Catherine Searle) and her NSF-DMS grant \#160178.   P. Rajan  and A. Siffert thank the Max Planck Institute for Mathematics in Bonn for support and providing excellent working conditions. A. Siffert was supported by a travel grant from DAAD \#464342.
This material is based in part upon work supported by the National Science Foundation under Grant No. DMS-1440140 while R. Perales and C. Searle were in residence at the Mathematical Sciences Research Institute in Berkeley, California, during the Spring 2016 semester.
\end{acknowledge}


\section{Preliminaries}\label{sec-background}
In this section we gather the definitions and tools we will need to prove Theorem A.
 In Subsection \ref{Orientation}, we present basic material on orientation in rectifiable spaces, in Subsection \ref{ssec-AS} we present the material we will need from 
Alexandrov geometry,
and in  Subsection \ref{ssec-ICS} we go over the 
definitions and basic tools we will need from the theory of Integral Current Spaces.

\subsection{Countably $\mathcal{H}^n$ Rectifiable Spaces}\label{Orientation}

In this subsection we establish properties related to orientation of a countably $\mathcal{H}^n$ rectifiable space, where $\mathcal{H}^n$ denotes the $n$ dimensional Hausdorff measure. We first recall the definition of such a space.

\begin{definition}[{\bf Countably $\mathcal{H}^n$ Rectifiable Space}] Given Borel measurable sets $A_i\subset \R^n$, we say that a metric space $X$ is  {\em countably $\mathcal{H}^n$ rectifiable} if and only if there
exist countably many Lipschitz maps $\varphi_i: A_i\rightarrow X$ such that $$\mathcal{H}^n( X\setminus \bigcup_{i=1}^{\infty}\varphi_i(A_i))=0.$$ 
\end{definition}

In defining an orientation on a countably  $\mathcal{H}^n$ rectifiable space, we first need an atlas.
We now recall the definition of an atlas for a countably  $\mathcal{H}^n$  rectifiable space, as given in Federer \cite{F}.

\begin{definition}[{\bf Atlas of Rectifiable Space}]\label{Atlas} For a countably  $\mathcal{H}^n$  rectifiable space, $X$, 
a bi-Lipschitz collection of charts $\{ A_i, \{\varphi_i\}\}$, where $\varphi_i: A_i\subset \R^n\rightarrow X$, is called an \emph{atlas} of $X$.
\end{definition}

An orientation on a countably  $\mathcal{H}^m$ rectifiable space is now defined via an oriented atlas as follows.

\begin{definition}[{\bf Oriented Atlas}] \label{Federer}
Suppose we have a countably $\mathcal{H}^n$ rectifiable space $X$ . Then we say that we have an {\em oriented atlas} if 
$$\det (\nabla(\varphi_i^{-1}\circ\varphi_j))>0$$
for all overlapping charts 
almost everywhere on  $\varphi_j(A_j)  \cap \varphi_i(A_i)$. 
\end{definition}

\begin{definition}[{\bf Orientation}]\label{CHmR} An {\em orientation} on a countably $\mathcal{H}^n$ rectifiable space $X$ is an equivalence class of atlases where two atlases, $\{A_i, \varphi_i  \}$, $\{\bar{A}_j, \bar{\varphi_j} \}$ are considered to be equivalent if their union is an oriented atlas.
\end{definition} 

\begin{remark} Given an orientation,  $[\{A_i, \varphi_i  \}]$,  we can choose a representative atlas, $\{A_i, \varphi_i  \}$, such that the charts are pairwise disjoint, $\varphi_i(A_i)  \cap \varphi_j (A_i) = \emptyset $, and the domains $A_i$ are precompact. We call such an oriented atlas a {\em preferred oriented atlas}.
\end{remark}


\subsection{Alexandrov Spaces}\label{ssec-AS}

Recall that a complete, finite dimensional length space $(X,d)$ is an \emph{Alexandrov space} if it has curvature bounded from below (see e.g. \cite{BBI}). Alexandrov spaces are 
known to be locally compact \cite{BBI}. Further, complete, locally compact, intrinsic metric spaces are {\em proper}, that is, any bounded closed set in such a space is compact (cf. \cite{PY}). Recall that a metric space is first countable. Let $X$ be an Alexandrov space, then for any $x\in X$ and for $r\in \mathbb{Z}^+$ consider the countable exhaustion of  $X$ by metric balls $B_r(x)$. Since $X$ is proper, each ball endowed with the subspace topology is precompact and first countable and hence second countable. Further, a metric space is second countable if and only it it is separable. Hence Alexandrov spaces are separable.

\smallskip

In an Alexandrov space, the unit tangent space to a point is replaced with the \emph{space of directions}, $\Sigma_x$, which is  
defined to be the completion of the 
space of geodesic directions at $x$. Observe that $\Sigma_x$ is itself a compact Alexandrov space with curvature bounded below by $1$. 
For $p, q\in X$, we will denote the set of all directions at $p$ corresponding to minimizing geodesics from $p$ to $q$ by $\Uparrow_p^q$. If there is a unique 
minimizing direction between $p$ and $q$, we will denote it by $\uparrow_p^q$.

\smallskip

We make the following distinction between types of points in an Alexandrov space, based on its space of directions. That is, we call a point $x\in X$ {\em regular} if $\Sigma_x$ is isometric to $S^{n-1}(1)$, the $(n-1)$-dimensional unit round sphere, and {\em singular} otherwise. 
We will denote the set of all regular points in $X$ by $R_X$ and the set of singular points by $S_X$. 

\smallskip

Next we introduce a technical tool for the introduction of a local coordinate system, the so-called strainers, originally defined in  \cite{BGP}.
We use $\sphericalangle(p, x, q)$ to denote the angle at $x$ of the geodesic triangle at $\Delta(pxq)\subset X$.
We use $\sphericalangle(\mu, \eta)$ to denote the angle between the two directions $\mu, \eta\in \Sigma_p$.
We use $\tilde{\sphericalangle}(p, x, q)$ to denote the comparison angle, that is the angle of the geodesic triangle $\tilde{\Delta}(pxq)$ in the simply connected $n$ dimensional Riemannian manifold with constant sectional curvature equal to $\kappa$, $\mathcal{M}^n_{\kappa}$.

\begin{definition}[\textbf{Strainers}]
\label{strainer}Let $X$ be an Alexandrov space. A point $x\in X$ is said to
be $( n,\delta) $--strained by the strainer $\left\{ \left(
a_{i},b_{i}\right) \right\} _{i=1}^{n}\subset X\times X$ provided that for
all $i\neq j$ we have%
\begin{equation*}
\begin{array}{ll}
\tilde{\sphericalangle}\left( a_{i},x,b_{i}\right) >\pi -\delta , \,\,\tilde{%
\sphericalangle}\left( a_{i},x,b_{j}\right) >\frac{\pi }{2}-\delta , \\ 
\tilde{\sphericalangle}\left( b_{i},x,b_{j}\right) >\frac{\pi }{2}-\delta,\,\,
\tilde{\sphericalangle}\left( a_{i},x,a_{j}\right) >\frac{\pi }{2}-\delta. 
\end{array}%
\end{equation*}

We say $B\subset X$ is an $(n,\delta)$--strained set with strainer $\left\{ \left(
a_{i},b_{i}\right) \right\} _{i=1}^{n}$ provided every point $x\in B$ is $(n,\delta)$%
--strained by $\left\{ \left(
a_{i},b_{i}\right) \right\} _{i=1}^{n}$. 
\end{definition}

Following  the convention in \cite{BBI}, if $\delta<1/100 n$, then we call an $(n,\delta)$-strainer an {\em $n$-strainer}.

\begin{definition}[{\bf Strainer number}]\cite{BBI}
The {\em strainer number} of an Alexandrov space $X$ is the supremum of numbers $n$ such that there exists an $n$-strainer in $X$.
\end{definition}

We will denote the set of all $\left( n,\delta \right) $-strained 
points of $X$ by $R_{X(n, \delta)}$. In general,  we will work with the set $R_{X(n, \delta)}$ rather than $X$ itself, as many properties of $X$ can be expressed in terms of $R_{X(n, \delta)}$. In particular, $R_{X(n, \delta)}$ satisfies a number of very useful properties, which we detail below.

\smallskip

The following Theorem from \cite{BBI} characterizes the set of $(n, \delta)$-strained points.
 \begin{theorem}[10.8.23 in \cite{BBI}] Let $X$ be an $n$-dimensional Alexandrov space. Then for every $\delta>0$ the set of $(n, \delta)$-strained points, that is,  $R_{X(n, \delta)}$, is an open dense set in $X$.
 \end{theorem} 
 
 Moreover, we can also show that $R_{X(n, \delta)}$ is path connected and hence  connected, as follows.
 
 \begin{lemma}\label{pconn}
 The set of $(n, \delta)$-strained points, $R_{X(n, \delta)}$, is path-connected.
 \end{lemma}
 \begin{proof} 
 
Let $x, y\in R_{X(n, \delta)}$. Consider $\gamma: [0, 1]\rightarrow X$, a geodesic from $x$ to $y$. By Theorem 1.1 in \cite{Pet1}, it follows that the spaces of directions along the interior of a geodesic are isometric to each other.
 Since both $x$ and $y$ are $(n, \delta)$-strained, we can find open neighborhoods around $x$ and $y$ that are also $(n, \delta)$-strained and 
 hence the interior points of $\gamma$ are $(n, \delta)$-strained and the result follows.
 \end{proof}

The theorem below shows that for $\delta$ sufficiently small,  an $(n,\delta)$-strained point has space of directions, $\Sigma_x$,  almost isometric to a round sphere. 
 
\begin{theorem}\label{almostisometry}\cite{BGP} Let $X$ be a complete $(n-1)$-dimensional space with curvature $\geq 1$ which has an $(n, \delta)$-strainer $\left\{ \left(a_{i},b_{i}\right) \right\} _{i=1}^{n}$. Then for small $\delta>0$, $X$ is almost isometric to the unit sphere $S^{n-1}(1)$, that is, there exists a homeomorphism $\tilde{f}:X\rightarrow S^{n-1}(1)$ such that 
$$||\tilde{f}(r)\tilde{f}(q)|-|rq||<\tau(\delta)|rq|$$ is satisfied for any $q, r\in X,$ where $\tau(\delta)$ is a function satisfying $\lim_{\delta\rightarrow 0} \tau(\delta)=0$.
\end{theorem}

In fact, by Corollary 9.6 from \cite{BGP},
if we again choose $\delta$ sufficiently small, $R_{X(n, \delta)}$ consists entirely of interior points.\begin{corollary}\cite{BGP}\label{interior} Let $X$ be a complete $(n-1)$-dimensional space with curvature $\geq 1$ which has an $(n, \delta)$-strainer $\left\{ \left(a_{i},b_{i}\right) \right\} _{i=1}^{n}$. 
Then, for sufficiently small $\delta>0$,  
\begin{enumerate}

\item $X$ has no boundary; and 
\item The set of $(n, \delta)$-strained points in an $n$-dimensional Alexandrov space with curvature bounded below are interior points.
\end{enumerate}
\end{corollary}

 Moreover, for every $\delta>0$, the set $R_X$ is contained in the set $R_{X(n, \delta)}$. In fact, $R_X$ is the intersection of 
 the $(n, \delta)$-strained points, as detailed in the theorem below.
 
\begin{theorem}\cite{BBI}\label{regulardense}
The set of regular points, $R_X$,  in an Alexandrov spaces is dense and moreover is an intersection of a countable collection of open dense sets. More precisely,
$R_X = \cap_{i\in\N}R_{X(n,1/i)}$.
\end{theorem}

 Using the fact that the set of regular points is the intersection of the sets $R_{X(n,1/i)}$, the following result is obtained in Otsu and Shioya \cite{OS}.
 
\begin{theorem}\cite{OS}\label{Hausdim} Let $X$ be an $n$-dimensional Alexandrov space.  The set $S_X$ of singular points in $X$ is of Hausdorff dimension $\leq n-1$.
\end{theorem}

\begin{corollary}\cite{BGP}\label{burst}
 The Hausdorff dimension of $X\setminus R_{X(n, \delta)}$ is less than or equal to $n-1$.
\end{corollary}

\begin{proof}
We have $R_X\subset R_{X(n, \delta)}$ and thus $X\setminus R_{X(n, \delta)} \subseteq S_X$.  
\end{proof}

For the special case of Alexandrov spaces without boundary we get better results.
The following lemma gives us a restriction on the codimension of the  the complement of $R_{X(n, \delta)}$.

\begin{lemma}\label{l:singcodim2}
Let $X$ be an Alexandrov space $X$ without boundary.
Then the set $X\setminus  R_{X(n, \delta)}$ is of Hausdorff codimension $2$ or greater.
\end{lemma}

\begin{proof}
Suppose that there exists an $(n-1)$-strained point.
 By Corollary 12.8 in \cite{BGP}, it follows that an $(n-1)$-strained point is also an $n$-strained point. 
 Recall that if a $(n,\delta)$-strained point exists, then the Hausdorff dimension of $X$ is at least $n$.
This establishes a contradiction to Theorem \ref{burst} and thus there exists no $(n-1)$-strained point.
\end{proof}

\begin{corollary}
The set of singular points, $S_X$, in an Alexandrov space without boundary has codimension $2$ or greater.
\end{corollary}
The proof follows in an analogous fashion to that of Corollary \ref{burst} and we leave it to the reader.

\smallskip

Now, for each  point $x \in R_{X,(n,\delta)}$, there exist bi-Lipschitz maps in a neighborhood of $x$. 
If $x\in R_X$ these Lipschitz constants can be made arbitrarily close to 1.
In particular, we have the following theorem (see Theorems 10.8.4, 10.8.18 and 10.9.16 in \cite{BBI}).

\begin{theorem}\cite{BBI}\label{bilip}
If $x\in X$ is an $n$-strained point and $n$ equals the local strainer number at $x$, then $x$ has a neighborhood which is bi-Lipschitz homeomorphic to an open region in $\R^n$.
A bi-Lipschitz homeomorphism is provided by distance coordinates associated with any $n$-strainer. Moreover the Lipschitz constants of this map and its inverse are not greater than $500 n$.\\
\end{theorem}

\begin{theorem}\cite{BBI}\label{bilip2}
Let $n \geq 1$ be an integer and $\varepsilon > 0$. Then there is $\delta >0$ such that every $(n,\delta)$-strained point in any $n$-dimensional 
Alexandrov space has a neighborhood which is  bi-Lipschitz  homeomorphic to an open region in $\mathbb R^n$. Moreover the Lipschitz constants 
of this map and its inverse are bounded by $1-\epsilon$ and $1+\epsilon$.
\end{theorem}

We can now use Theorem \ref{bilip2} to construct an atlas for an Alexandrov space that will be compatible with the 
definition of an atlas for a countably $\mathcal{H}^n$ rectifiable space (see Definition \ref{Atlas}). 

\begin{theorem}\label{atlas} 
Given $\epsilon>0$, there exists an atlas $\{ A_i, \{\varphi_i\}\}$ of $X$ such that $\varphi_i:A_i\rightarrow R_{X(n, \delta)}$ are bi-Lipschitz with uniform bi-Lipschitz constants bounded between $1-\epsilon$ and $1+\epsilon$.
Furthermore, the images of the $\varphi_i$ can be made to be disjoint.
\end{theorem}

\begin{proof}
Since open subsets of separable metric spaces are separable, it follows that $R_{X(n, \delta)}$ is separable.
Let $\{x_i\}_{i\in I}$ be a countable, dense collection of points in $R_{X(n, \delta)}$. Around each $x_i\in R_{X(n, \delta)}$,
construct an open neighborhood $U_{x_i}$ and a bi-Lipschitz map $f_{x_i}:U_{x_i} \rightarrow \R^n$, as in Theorem \ref{bilip2}. 
The union $\cup_{i\in I} U_i$ covers all of $R_{X(n, \delta)}$ since the $\{x_i\}_{i\in I}$ are dense in $R_{X(n, \delta)}$.

We can now construct a countable collection of maps on the set of $(n, \delta)$-strained points, $R_{X(n, \delta)}$.
We define
$\varphi_i : f_{x_i}(U_{x_i}) \to X$ by 
$$\varphi_i =  f^{-1}_{x_i}.$$ Setting  $A_1=f_{x_1}(U_{x_1})$  and $A_i=f_{x_i}(U_{x_i}  \setminus  \bigcup_{j=1}^{i-1} U_{x_j})$ for $i > 1$ 
we obtain Borel sets and can restrict the $\varphi_i$ to these to obtain bi-Lipschitz functions with disjoint images. Thus we have constructed an atlas, $\{A_i, \{\varphi_i\}\}$, which is countable, whose maps are bi-Lipschitz, whose images are disjoint and such that almost every point in $X$ is covered by this atlas.
\end{proof}

In order to make sure the charts in an atlas for an Alexandrov space satisfy the Federer definition of orientability (see Definition \ref{Federer}), we will need to be able to differentiate distance functions. However, in our case, we only require the directional differentiability of distance functions. In particular, for $f=\dist_p$,  we have 
\begin{equation}\label{dist}
d_pf(v_s)=-\min_{\xi\in \Uparrow_p^q} \langle v_s, \xi\rangle\coloneqq- \langle v_s,  \Uparrow_p^q \rangle.
\end{equation}
We also recall the following lemma of Lytchak \cite{Ly}:

\begin{lemma}\label{composition} Let $f: X \rightarrow Y$ and $g: Y\rightarrow Z$ be Lipschitz maps with $f(x)=y$ and $g(y)=z$. If $f$ is differentiable at $x$ and $g$ is differentiable at $y$, then $g\circ f$ is differentiable at $x$ with differential 
$D_x(g\circ f)=D_yg\circ D_xf$.
\end{lemma}

Finally, we recall the Bishop and Bishop-Gromov inequalities for Alexandrov spaces here, as they will be needed for the proof of Theorem A.

\begin{theorem}[{\bf Bishop inequality}] \cite{BBI}\label{bishop Ineq}
Let X be an $n$-dimensional Alexandrov space of curvature $\geq \kappa$,  then for every $x \in X $ and every $r>0$
$$\mathcal H ^n(B_r(x))\leq V_{\kappa}(r),$$
where $\mathcal H ^n(B_r(x))$ is the $n$-dimensional Hausdorff measure of the ball of radius $r>0$ centered at x and $V_{\kappa}(r)$ is the volume of a ball of radius r in the space form $\mathcal{M}^n_{\kappa}$.
\end{theorem}

\begin{theorem}[{\bf Bishop-Gromov inequality}]\cite{BBI}\label{Grombishop Ineq}
Let $X$ be an $n$-dimensional Alexandrov space of curvature $\geq \kappa$. Then for every $x \in X$ the ratio
$$ \frac{\mathcal H ^n(B_r(x))}{V_{\kappa}(r)}$$
is nonincreasing in $r$. In other words, if $R>r>0$,
then  $$ \frac{\mathcal H ^n(B_R(x))}{V_{\kappa}(R)} \leq   \frac{\mathcal H ^n(B_r(x))}{V_{\kappa}(r)}.$$
 \end{theorem}


\subsection{Integral Current Spaces}\label{ssec-ICS}

\vskip .5cm

The aim of this subsection is to review the definition of $m$-dimensional integral current space as defined by Sormani-Wenger \cite{SW}. In order to do so we review previous definitions given in \cite{AK} such as $m$-dimensional current, $m$-dimensional integer rectifiable current and $m$-dimensional integral current. Moreover, we list basic results about currents and current spaces. 
\smallskip

Let $(Z, d)$ be a metric space. Define $\mathcal{D}^m(Z)$ to be the collection
of $(m+1)$-tuples of Lipschitz functions where the first entry function is bounded, that is, 
\begin{equation*}
\mathcal{D}^m(Z): =\left\{ (f,\pi)=\left(f,\pi_1 ..., \pi_m\right)\, |\, f, \pi_i: Z \to \R \text{ Lipschitz and } f \,\text{ is bounded}\right\}.
\end{equation*}

If $Z= \R^m$, then a basic example of an $m$-dimensional current there is $\Lbrack h \Rbrack:  \mathcal{D}^m(\R^m) \to \R$
 where   $h: A \subset \R^m \to \Z$ is an $L^1$ function, and  
 $\Lbrack h \Rbrack$ is given by
\bdm \label{def-current-from-function}
\Lbrack h \Rbrack \left(f, \pi\right) = \int_{A \subset \R^m}  h f \det\left(\nabla \pi_i\right) \, dx.
\edm

For easy reference, we begin with a list of notations that will be explained below:
\begin{align*}
 \curr_m(Z)=&\{m-\textrm{dimensional currents on } Z\},\\
\intrectcurr_m\left(Z\right)=&\{m-\textrm{dimensional integer rectifiable currents on } Z\},\\
\intcurr_m\left(Z\right)=&\{m-\textrm{dimensional integral currents on } Z\}.
\end{align*}
We  note that $$\intcurr_m\left(Z\right)
\subset \intrectcurr_m\left(Z\right)\subset \curr_m(Z).$$

We are now ready to define an $m$-dimensional current, its mass, and the associated operators of boundary, pushforward and restriction.
\begin{definition}[{\bf Current}]\label{defn-current}
Let $Z$ be a complete metric space. A multilinear functional $T:\mathcal{D}^m(Z) \to \R$ is called an { \em $m$-dimensional  current} if it satisfies the following:
\begin{enumerate} 
\item  If there is an $i$ such that $\pi_i$ is constant on a neighborhood of $\{f\neq0\}$ then $T(f, \pi)=0$.
\item $T$ is continuous with respect to the pointwise convergence of the $\pi_i$ for  $\Lip(\pi_i)\le 1$.
\item  There exists a finite Borel measure $\mu$ on $Z$ such that for all $(f,\pi)\in \mathcal{D}^m(Z)$
\bdm\label{def-AK-current-iii}
|T(f,\pi)| \le \prod_{i=1}^m \Lip(\pi_i)  \int_Z |f| \,d\mu .
\edm
\end{enumerate}
\end{definition}

In the following definitions,  $Z$ will always denote a complete metric space.
\begin{definition}[{\bf Mass}] 
The smallest Borel measure that satisfies Part $(3)$ in the previous definition is called the mass measure of $T$ and is denoted by $\|T\|$. The {\em mass} of $T$ is given by
\bdm \label{def-mass-from-current}
\mass (T) := \| T \| \left(Z\right) = \int_Z \, d\| T\|.
\edm
\end{definition}

\begin{definition}[{\bf Boundary}]  Let $T\in \curr_m(Z)$. The {\em boundary of $T$}, denoted by $\partial T$, is the function $\bdry T : \mathcal{D}^{m-1}(Z) \to \R$ given by
\begin{equation*}
\partial T \left(f, \pi_1,..., \pi_{m-1}\right) = T \left(1, f, \pi_1,..., \pi_{m-1}\right).
\end{equation*}
\end{definition}

\begin{remark}
The boundary of a current need not be a current itself. Indeed, by Remark 2.47 in \cite{SW} the following holds: let $\phi :K\subset\R^m\rightarrow Z$ be a chart where $K$ is compact and $Z$ a metric space.
Then $\partial\phi_{\#}\Lbrack 1_K\Rbrack$ is a current if and only if $K$ has finite perimeter.  
\end{remark}

\begin{definition}[{\bf Pushforward}]\label{defn-push}
Let $T\in \curr_m(Z)$ and $\varphi:Z\to Z'$ be a
Lipschitz map.  The {\em pushforward of $T$},  $\varphi_\# T: \mathcal{D}^{m}(Z') \to \R$, is the 
function given by
\bdm \label{def-push-forward}
\varphi_\#T(f,\pi_1,...,\pi_m)=T(f\circ \varphi, \pi_1\circ\varphi,..., \pi_m\circ\varphi).
\edm
\end{definition}
\noindent Note that, by construction, $\varphi_\#$ commutes with the boundary operator, that is, 
$$\varphi_\#(\partial T)=\partial(\varphi_\#T).$$ Recall also that  in \cite{AK} it was proven that the pushforward of a current is a current, that is, $\varphi_\# T \in \curr_m(Z')$.

\begin{definition}[{\bf Restriction}] Let $T\in \curr_k(Z)$ and let $\omega=(g, \tau_1, \hdots, \tau_m)\in \mathcal{D}^m(Z)$, 
with $m\leq k$ ($\omega=g$ if $m=0$). The {\em restriction of $T$ to $\omega$} is the $(k-m)$-dimensional current in $Z$, denoted by $T\rstr\omega$, given by  
$$T\rstr \omega(f, \pi_1, \hdots, \pi_{k-m})=T(fg, \tau_1, \hdots, \tau_m, \pi_1, \hdots, \pi_{k-m}).$$
\end{definition}

\begin{definition}[{\bf Integer Rectifiable Current}]  \label{def-integercur} 
Let $T\in \curr_m(Z)$. We say that $T$ is an $m$-dimensional {\em integer rectifiable
current}, that is, $T\in\intrectcurr_m\left(Z\right)$, if it has a {\em current parametrization}, consisting of parametrizations and weight functions, $\left(\{\varphi_i\}, \{\theta_i\}\right)$, satisfying the following conditions.
\begin{enumerate}
\item  The set of parametrizations $\varphi_i:A_i\subset \mathbb{R}^m \to Z$ is a countable collection of bi-Lipschitz maps
such that all $A_i$ are precompact Borel measurable with pairwise disjoint images.
\item  The weight functions, $\theta_i\in L^1\left(A_i,\N \right)$, are defined so that 
 the following equalities hold:
\bdm\label{param-rep}
T = \sum_{i=1}^\infty \varphi_{i\#} \Lbrack \theta_i \Rbrack \quad\text{and}\quad \mass\left(T\right) = \sum_{i=1}^\infty \mass\left(\varphi_{i\#}\Lbrack \theta_i \Rbrack\right).
\edm
\end{enumerate}
\end{definition}

\noindent Using the above definition, the mass measure can then be rewritten as
\bdm
\|T\| = \sum_{i=1}^\infty \|\varphi_{i\#}\Lbrack \theta_i \Rbrack \|.
\edm

The following lemma from \cite{SW} gives us criteria for when two $m$-dimensional integer currents are equal. Before stating the lemma, we need the following definition for the weight of a current, $T$.
Let $\theta_T: Z \to \N \cup \{0\}$ be the $L^1$ function, called the {\em weight} of $T$, given by
\begin{equation}\label{e:1}
\theta_T= \sum_{i=1}^\infty \theta_i\circ\varphi_i^{-1}\One_{\varphi_i\left(A_i\right)}.
\end{equation}

\begin{lemma}\cite{SW}\label{lem-equalcurr} 
Let $T$ and $T'$ be two $m$-dimensional integer currents defined on a complete metric space $Z$ with current parametrizations $\left(\{\varphi_i\}, \{\theta_i\}\right)$ and $\left( \{\varphi'_i\}, \{\theta'_i\}\right)$, respectively. 
Let $A_i$ and $A'_i$ be the domains of the charts $\varphi_i$ and $\varphi'_i$, respectively.
Then $T=T'$ if the following conditions are satisfied.
\begin{enumerate}
\item The symmetric difference between $\cup \varphi_i(A_i)$ and $\cup \varphi'_i(A'_i)$ has zero $m$-dimensional Hausdorff measure. 
\item In all overlapping sets we have
\begin{equation*}
\det(\nabla(\varphi^{-1}_i \circ \varphi'_j)) > 0 \text{   and    } 
\det(\nabla(\varphi'^{-1}_i \circ \varphi_j)) > 0.
\end{equation*}
\item  The functions $\theta_T$ and $\theta_{T'}$ agree $\mathcal{H}^m$  almost everywhere on $Z$.
\end{enumerate}
\end{lemma}

\begin{definition}[{\bf Integral Current}]\label{defn-integralcur}
Let $T\in \intrectcurr_m(Z)$.
Then we call  $T$  an $m$-dimensional {\em integral current}, that is $T\in \intcurr_m(Z)$, provided  $\partial T$ is an $(m-1)$-dimensional current, that is $\partial T\in \curr_{m-1}(Z)$. 
\end{definition}

Recall that the $m$-dimensional {\em density} of a Borel measure $\mu$ at $z\in Z$ is defined as 
$$\lim_{r\to 0} \frac{\mu(B_r(z))}{\omega_m r^m},$$
where $\omega_m$ denotes the volume of the unit ball in $\R^m$. We now define the {\em lower density} of a Borel measure, which we will use to define the {\em canonical set} of a current. With the canonical set defined, we will arrive at a definition of an {\em integral current space}.

\begin{definition}[{\bf Lower Density}]
The $m$-dimensional {\em lower density}, $\Theta_{*m}(\mu, z)$, of a Borel measure $\mu$ at $z\in Z$ is defined as 
$$\Theta_{*m}(\mu, z)= \liminf_{r\to 0} \frac{\mu(B_r(z))}{\omega_m r^m},$$
where $\omega_m$ denotes the volume of the unit ball in $\R^m$.
\end{definition}

We now define both the canonical set and arrive at a definition of an integral current space.

\begin{definition}[{\bf Canonical Set}]\label{defn-set}
Let $T \in \curr_m(Z)$. The {\em canonical set of $T$}, denoted by $\set(T)$, is defined as
\bdm
\set(T)=\{ z\in Z: \, \liminf_{r\to 0} \frac{\|T\|(B_r(z))}{\omega_m r^m} >0\}
\edm
where $\omega_m$ denotes the volume of the unit ball in $\R^m$.
\end{definition}

In the next lemma we see that the mass measure of an integral current $T$  is concentrated in $\set(T)$.

 \begin{lemma}\cite{AK}\label{lemma-weight}
Let $T \in \intrectcurr_m\left(Z\right)$ with current parametrization $\left(\{\varphi_i\}, \{\theta_i\right\})$. Then there is a function
\bdm \label{eqn-lem-weight-lambda}
\lambda:\set(T) \to [m^{-m/2}, 2^m/\omega_m]
\edm
satisfying
$$\Theta_{*m}(\| T\|, x)=\theta_T(x)\lambda(x),$$
such that
\bdm \label{eqn-lem-weight-2}
\|T\|=\theta_T \lambda \mathcal{H}^m \rstr \set(T),
\edm
where $\omega_m$ denotes the volume of an unitary ball in $\R^m$ and $\theta_T$ is defined as in Equation (\ref{e:1}).
\end{lemma}

We can now define an integral current space. Recall that this is the type of space we construct in Theorem A and consider sequences of them in Theorem B.  
\begin{definition}[{\bf Integral Current Space}]\label{defn-intcurrspace}
Let $\left(Z,d\right)$ be a metric space and $T \in \intcurr_m(\bar{Z})$. If $\set\left(T\right)=Z$ then $(Z,d,T)$ is called an $m$-dimensional {\em integral current space}. 
\end{definition}

Finally, we give the definition of intrinsic flat distance between two $m$-dimensional integral currents \cite{SW}. 
This distance together with the Gromov-Hausdorff distance is used in Li-Perales's Theorem \ref{LP} and Theorem B. 

\begin{definition}[{\bf Intrinsic Flat Distance}]
 Let $(X_i,d_i,T_i)$, $i=1,2$, be two $m$-dimensional integral current spaces. 
 Its {\em intrinsic flat distance} is defined as 
 \begin{equation}
 d_\frak{F}((X_1,d_1,T_1), (X_2, d_2, T_2))= \inf \{ \mass(U) + \mass(V)\},
 \end{equation}
where the infimum is taken over all complete metric spaces $Z$, all isometric embeddings $\varphi_i: X_i \to Z$ and all currents $U \in \intcurr_m\left(Z\right)$
and $V \in \intcurr_{m+1}\left(Z\right)$ that satisfy 
\begin{equation}
\varphi_{1\#}(T_1) - \varphi_{2\#}(T_2) = U  + \partial V.
\end{equation}
\end{definition}

Sormani-Wenger proved that $d_\frak{F}$ is a distance in the class of $m$-integral current spaces whose $\set$ is precompact. 
We recall that $d_\frak{F}((X_1,d_1,T_1), (X_2, d_2, T_2))=0$ if and only if there is an isometry $\varphi: X_1 \to X_2$ that preserves 
orientation, ie, $\varphi_{\#}(T_1)=T_2$. 

In general, the intrinsic flat limit of a sequence of $m$-dimensional integral currents 
can exist without the Gromov-Hausdorff limit having to exist. See \cite{SW} for examples. But when the Gromov-Hausdorff limit exist and the 
mass of the currents in the sequence and the mass of their boundaries are uniformly bound then the intrinsic flat limit exists and is 
contained in the Gromov-Hausdorff limit (Theorem 3.20 in \cite{SW}). Note that the intrinsic flat limit either has $m$ Hausdorff dimension or collapses 
to what is called the zero integral current space.

\subsection{Homology Theory of Integral Currents}

We define a {\em Lipschitz $k$-simplex} to be a Lipschitz map $\sigma: \Delta^k\rightarrow X$. 
Letting $C_k(X)$ denote the usual group of $k$-singular chains on $X$,  then $C_k^{Lip}(X)$ is the subgroup of  $C_k(X)$ with basis the singular Lipschitz simplices.

Yamaguchi defines the notion of a {\em locally Lipschitz contractible space} in \cite {Y} to be a metric space for which small metric balls are contractible to a point via a Lipschitz homotopy.  It is clear, for example, that Alexandrov spaces are locally Lipschitz contractible. He then shows that 
the singular homology of a locally Lipschitz contractible space, $X$, is isomorphic to its Lipschitz homology, that is,  $H_*^{\textrm{sing}}( X; \mathbb{Z})$ is isomorphic to $H^{\textrm{Lip}}_*(X; \mathbb{Z})$.

In \cite{M1}, Mitsuishi defines an integral current on a locally Lipschitz contractible space as follows.
Given a Lipschitz $k$-simplex $f : \Delta^k \to X$, we define the integral $k$-current $[T]\in \intcurr_k^{\textrm{c}}(X)$ to be $T=f_{\#} \Lbrack \Delta^k\Rbrack$, where $\intcurr_k^{\textrm{c}}$ denotes the integral currents on $X$ with compact support. The $\mathbb Z$-linear extension 
$$[\cdot]: C_k^{Lip}(X) \to \intcurr_k^{\textrm{c}}\left(X\right)$$ then gives us a chain map and a chain complex on $X$, denoted by $\intcurr_{\bullet}^{\textrm{c}}(X)$.  We let  $H_*(\intcurr_{\bullet}^{\textrm{c}}(X); \mathbb{Z})$ denote the corresponding homology theory. Finally, Theorem 1.3 of \cite{M1} (see also Theorem 9.3 in Mitsuishi \cite{M2}) proves that  
$$H_*^{\textrm{sing}}( X; \mathbb{Z})\cong H^{\textrm{Lip}}_*(X; \mathbb{Z})\cong H_*(\intcurr_{\bullet}^{\textrm{c}}(X); \mathbb{Z}).$$

 Note that, by definition, $H_k(\intcurr_k^{\textrm{c}}(X); \mathbb{Z})=\ker(\partial_k)/\textrm{Im}(\partial_{k+1})$,
 that is, the  $k$th Integral Current homology 
corresponds to the quotient of the $k$-dimensional integral currents whose boundary is $0$  mod out by the images of the $k+1$-dimensional integral currents.

Finally,  it is proven in Theorem  1.17 \cite{M2} that $\textrm{Im}(\partial_{k+1})$ is trivial for $k=n$, thus the following holds. 

\begin{theorem}\cite{M2} \label{thm-Mitsuishi}
Let $X$ be an $n$-dimensional closed,  orientable Alexandrov space. Then, 
\begin{equation*}
\{ T \in \intcurr_n\left(X\right) | \, \bdry T =0 \} = H_n(\intcurr_n\left(X\right); \mathbb{Z}) \cong H_n^{\textrm{sing}}(X; \mathbb{Z}) \cong \mathbb Z.
\end{equation*} 
In particular, there is a non-trivial integral $n$-current on $X$.
\end{theorem}

\bigskip

\section{Orientation of Alexandrov Spaces Via Strainers}\label{3}

The goal of this section is to prove that on an $n$ dimensional oriented
Alexandrov space without boundary, using the bi-Lipschitz homeomorphisms of
Theorem \ref{bilip}, we can construct an oriented atlas on the set of $(n,
\delta)$-strained points, $R_{X(n, \delta)}$, in the countably $\mathcal{H}%
^n $ rectifiable space sense. Since $R_{X(n, \delta)}$ is open and dense in $%
X$, this will give us an oriented atlas of $X$ in the countably $\mathcal{H}%
^n$ rectifiable space sense.


\subsection{Defining a topological orientation on $R_{X(n, \protect\delta)}$}

\label{sstop}
\noindent We 
recall the definition of an orientation for a topological
space.

\begin{definition}[\textbf{$\mathbb{Z}$ Orientation System}]
\label{Ori} A \emph{$\mathbb{Z}$ orientation system} for a topological
manifold $X$ consists of the following two elements:
\begin{enumerate}
\item An open cover $\{U_i\}$ of $X$;
\item For each $i$, a local orientation $\alpha_i\in
H_n(X, X\setminus U_i)$  of $X$ along $U_i$ such
that if $x\in U_i\cap U_j$, then 
\begin{equation*}
\iota_{x}^{U_i}(\alpha_i)=\iota_{x}^{U_j}(\alpha_j),
\end{equation*}
where
\begin{equation*}
\iota_x^U: H_n(X, X\setminus U)\rightarrow H_n(X, X\setminus x)
\end{equation*}
is the canonical homomorphism induced by inclusion.
\end{enumerate}
\end{definition}

\begin{definition}[\textbf{$\mathbb{Z}$-orientable}]\label{Z-o}
\label{topori} 
We will say a space $X$ is orientable in the \textit{%
topological sense}, if it has a $\mathbb{Z}$ orientation system.
\end{definition}

 In order to talk about orientability for Alexandrov spaces, we
must first understand what non-orientability means. We distinguish between two important cases: we call an
Alexandrov space \emph{locally orientable} if every point has an orientable
neighborhood, and \emph{locally non-orientable} otherwise. Alexandrov spaces are unlike manifolds in that they
can have arbitrarily small neighborhoods which do not admit an orientation (see Petrunin \cite{Pet1}).
In particular, if $x \in X$ has a non-orientable space of directions $%
\Sigma_x$, then no neighborhood of $x$ is orientable. We call such a space 
\emph{locally non-orientable}. Equivalently, we see that a point is \emph{%
locally orientable} if its space of directions $\Sigma_x$ is orientable.

On compact Alexandrov spaces, one uses singular cohomology with integer coefficients (cf. Grove, Petersen \cite{GP})
to study orientability.
It is easy to see by excision and Perelman's Stability theorem \cite{P} (see
also \cite{K1}), that $H^{n}(X,X\setminus \{x\})\cong H^{n-1}(\Sigma _{x})$.
Thus, if $H^{n-1}(\Sigma _{x})\cong {\mathbb{Z}}$, then we say that $X$ is \emph{%
locally orientable at }$x$, and a choice of generator for $H^{n-1}(\Sigma
_{x})$ is called a \emph{local orientation} at $x$. Using Theorem \ref{almostisometry}, the following Lemma is immediate.
\begin{lemma} Let $x\in X$ be an $(n,\delta)$-strained point with strainer $\left\{ \left(a_{i},b_{i}\right) \right\} _{i=1}^{n}$ for a point $x\in X$. Then for sufficiently small $\delta>0$, $R_{X(n, \delta)}$ is locally orientable at $x$.
\end{lemma}
 We define orientability of a closed Alexandrov space, 
$X$, in terms of the existence of a fundamental class, that is, $X$ is {\em orientable} if for every $x\in X$, $H^{n}(X, X\setminus
\{ x \}) \rightarrow H^n(X)\cong {\mathbb{Z}}$ is an isomorphism. Recently, Mitsuishi \cite{M2} has shown that closed, orientable Alexandrov spaces satisfy Poincar\'e Duality, so 
we can equivalently define orientability via a top class in {\it homology}.
With this definition of orientability, we can then show that $R_{X(n, \delta)}$ is orientable.



\begin{lemma}
\label{strained-orientable} Let $X$ be an $n$-dimensional Alexandrov space.
Given $\delta>0$, if $X$ is orientable in the topological sense, then $%
R_{X(n, \delta)}$ is also orientable in the topological sense.
\end{lemma}

\begin{proof}
Recall first that $R_{X(n, \delta)}$ is open and dense in $X$. If $X$ is
orientable in the sense of Definition \ref{Z-o}, then we have an open cover $%
\{ U_i\}$ of $X$ such that the local orientations coincide in the
intersection of any two of the sets in the cover.
\end{proof}

\begin{remark}
Note that in \cite{HS}, Harvey and Searle show that an Alexandrov space is
orientable if and only $X^{(n)}$, the top stratum, itself a topological
manifold, is orientable. We have shown that if $X$ is orientable, then $%
R_{X(n, \delta)}\subset X^{(n)}$ is also orientable. In fact, our proof of
Lemma \ref{strained-orientable} shows that  if $X^{(n)}$ is orientable then $R_{X(n, \delta)}$ is. 
However, the reverse implication seems difficult
to prove directly, since $X^{(n)}\setminus R_{X(n, \delta)}$ may not even be
a CW complex.
 \end{remark}

We now show that a topological orientation on an Alexandrov space  $X$ without boundary induces an orientation in
the sense of Federer. 

\begin{theorem}
\label{orientation-main} Given an $n$-dimensional Alexandrov space $X$
without boundary, oriented in the topological sense, then the set $R_{X(n,
\delta)}$ is oriented in the sense of countably $\mathcal{H}^n$ rectifiable
spaces, that is, there exists a bi-Lipschitz collection of charts $\{ A_i,
\{\varphi_i\}\}$, where $\varphi_i: A_i\subset {\mathbb{R}}^n\rightarrow
R_{X(n, \delta)}$ such that 
\begin{equation*}
\det (\nabla(\varphi_i^{-1}\circ\varphi_j))>0
\end{equation*}
for all overlapping charts almost everywhere on $\varphi_j(A_j) \cap
\varphi_i(A_i)$.
\end{theorem}

Before we begin the proof of Theorem \ref{orientation-main}, we need Lemma 1.4 from \cite{PW}, which shows us that a strainer is very close to being orthogonal. 

\begin{lemma}\cite{PW}
\label{CF1.4} 
\label{angle convergence}Let $B\subset X$ be $( l,\delta) $%
--strained by $\left\{ \left( a_{i},b_{i}\right) \right\} _{i=1}^{l}.$ For
any $x\in B$ and $i\neq j,$ 
\begin{equation*}
\begin{array}{ll}
\pi -\delta <\sphericalangle \left( a_{i},x,b_{i}\right) \leq \pi , & \frac{%
\pi }{2}-\delta <\sphericalangle \left( a_{i},x,b_{j}\right) <\frac{\pi }{2}%
+2\delta,\text{ and} \\ 
\frac{\pi }{2}-\delta <\sphericalangle \left( b_{i},x,b_{j}\right) <\frac{%
\pi }{2}+2\delta, & \frac{\pi }{2}-\delta
<\sphericalangle \left( a_{i},x,a_{j}\right) <\frac{\pi }{2}+2\delta.%
\end{array}%
\end{equation*}
Moreover, the same result holds for the comparison angles.
\end{lemma}

The following lemma allows us to create the required charts for the proof of
Theorem \ref{orientation-main}.

\begin{lemma}
\label{lemma3.9} For $x \in R_{X(n, \delta)}$ and $\epsilon>0$, let $\{(a_{i},\,b_{i})%
\}_{i=1}^{n}$ and $\{(c_{i},\,d_{i})\}_{i=1}^{n}$ be $( n,\delta) $--strainers for a neighborhood $U_x\subset X$ of $x$. Let 
\begin{equation*}
\phi _{a},\phi _{c}:U_{p}\rightarrow \mathbb{R}^{n},
\end{equation*}%
be given, as in Theorem \ref{atlas} by%
\begin{equation*}
\phi _{a}(y)=(\func{dist}(a_{1},y),\cdots ,\func{dist}(a_{n},y))\text{ and }%
\phi _{c}(y)=(\func{dist}(c_{1},y),\cdots ,\func{dist}(c_{n},y)),
\end{equation*}%
where $y\in U_x$ and the bi-Lipschitz constants of $\phi_a$ and $\phi_c$ are bounded between $1-\epsilon$ and $1+\epsilon$.
Then

\begin{enumerate}
\item The transition functions 
\begin{equation*}
\phi _{a}\circ \phi _{c}^{-1}:\phi _{c}\left( U_{x}\right) \rightarrow \phi
_{a}\left( U_{x}\right)
\end{equation*}%
\begin{equation*}
\phi _{c}\circ \phi _{a}^{-1}:\phi _{a}\left( U_{x}\right) \rightarrow \phi
_{c}\left( U_{x}\right)
\end{equation*}
are differentiable almost everywhere.
\item If $v\in \Sigma_x$ is a direction where both $\phi _{a}$ and $\phi_c$ are 
differentiable and orientation preserving as in Part ($2$), then the change of basis matrix, $M$, of $d\left( \phi _{a}\circ \phi
_{c}^{-1}\right) $ with respect to the standard basis satisfies%
\begin{equation*}
1-\tau(\delta)\leq \det(M)\leq 1+\tau(\delta),
\end{equation*}
where $\tau(\delta)$ is a function satisfying $\lim_{\delta\rightarrow 0} \tau(\delta)=0$.

\item If $R_{X(n,\delta )}$ is oriented, then after possibly interchanging $%
a_{1}$ with $b_{1}$, $\phi _{a}$ is orientation preserving.
\end{enumerate}
\end{lemma}

\begin{proof}
Note first that by Theorem 5.4 \cite{BGP} (cf. Theorem 10.8.18 of \cite{BBI}) both $\phi
_{a}$ and $\phi _{c}$ are bi-Lipshitz and hence differentiable almost
everywhere and we obtain Part ($1$). 

To prove Part ($2$), recall from Definition 11.3 in \cite{BGP} that the component functions $\left( \phi _{a}\right) _{i}$
and $\left( \phi _{c}\right) _{i}$ of $\phi _{a}$ and $\phi _{c}$ are
directionally differentiable almost everywhere. By definition of the derivative of a distance function (see Display \ref{dist}), it follows that  the $i$th components of the directional derivatives of $\phi_a$ and $\phi_c$ are given almost everywhere by 
\begin{equation*}
D_{v}\left( \phi _{a}\right) _{i}=-\cos (\sphericalangle (\Uparrow
_{x}^{a_{i}},v))\text{ and }D_{v}\left( \phi _{c}\right) _{i}=-\cos
(\sphericalangle (\Uparrow _{x}^{c_{i}},v)),
\end{equation*}
(cf. Example 11.4 in \cite{BGP}).

Further, it follows from the definition of the directional derivative that for almost every $y\in U_x$ we have
\begin{equation*}
(d\phi _{a})_{y}\left( \Uparrow
_{y}^{a_{i}}\right) = -e_{i}.
\end{equation*}
Note then that 
$$ d( \phi _{a}\circ \phi _{c}^{-1})_{\phi
_{c}( y)}(e_i)=\left[\cos \sphericalangle (\Uparrow _{y}^{a_{i}},\Uparrow
_{y}^{c_{j}})\right]_{i,j}(e_i),$$
and hence $$M=\left[\cos \sphericalangle (\Uparrow _{y}^{a_{i}},\Uparrow
_{y}^{c_{j}})\right]_{i,j}.$$

However, by the last inequality of Lemma \ref{CF1.4}, we have for every $y\in U_{x}$ and for $i\neq j$ 
\begin{equation*}
\left\vert \sphericalangle \left( \Uparrow _{y}^{a_{i}},\Uparrow
_{y}^{a_{j}}\right) -\frac{\pi }{2}\right\vert <2\delta \, \, \textrm{ and } \, \, \left\vert \sphericalangle \left( \Uparrow _{y}^{c_{i}},\Uparrow
_{y}^{c_{j}}\right) -\frac{\pi }{2}\right\vert <2\delta.
\end{equation*} 
Thus, we may choose oriented, orthonormal bases, $\{v_i\}_{i=1}^n$,  and  $\{w_i\}_{i=1}^n$, for $\mathbb{R}^n \supseteq \Sigma_y$,    such that
$$|\sphericalangle(\Uparrow_y)^{a_{i}}, v_i)| < 2\delta \,\,  \mbox{ and } \,\, |\sphericalangle (\Uparrow_{y}^{c_{i}}, w_i) |< 2\delta,$$
noting that the vectors $v_i$ and $w_j$  need not all be elements of $\Sigma_y$. The  change of basis matrix from $\{v_i\}_{i=1}^n$ to $\{w_i\}_{i=1}^n$, with respect to the standard basis in $\mathbb{R}^n$,  is given by $\left[\cos \sphericalangle (v_i,\, w_j) \right]_{i,j}$ and has determinant $1$.

Applying the triangle inequality twice, we have
$|\sphericalangle( \Uparrow_{y}^{a_{i}}, \Uparrow_{y}^{c_{j}}) - \sphericalangle (v_i,\, w_j) |< 4\delta$. 
Hence $$1-\tau(\delta)\leq \det(M)\leq 1+\tau(\delta)$$ and Part ($2$) is proven.

We now prove Part ($3$). Note by Lemma \ref{CF1.4} , it follows that 
\begin{equation}\label{e1}
\pi-2\delta-\sphericalangle(\Uparrow_y^{a_1}, \Uparrow_y^{a_j})\leq\sphericalangle(\Uparrow_y^{b_1}, \Uparrow_y^{a_j})\leq \pi +4\delta -\sphericalangle(\Uparrow_y^{a_1}, \Uparrow_y^{a_j}).
\end{equation}

Since we have a consistent top class at every point, we have a global orientation on $X$.  If the orientation induced by the map $\phi _{a}$ coincides with the global orientation,  we do not change the map $\phi _{a}$. If it is not, then by interchanging $a_1$ with $b_1$, it follows immediately by Inequality \ref{e1} that the newly defined $\phi_a$ will be orientation preserving.
\end{proof}

\begin{lemma}
\label{strained-orientable-grad} Let $X$ be an Alexandrov space. Given $%
\delta >0$ as in Theorem \ref{almostisometry}, if $X$ is oriented in the
topological sense, then on $R_{X(n,\delta )}$ there exists charts given by
strainers $\{(U_{i},\,\phi _{i})\}$ such that at almost every point of $%
U_{i}\cap U_{j},$ $$\det (\nabla (\phi _{i}\circ \phi _{j}^{-1}))>0.$$
That is, there exists an oriented atlas on $X$ satisfying Definition \ref{Federer}.
\end{lemma}

\begin{proof}
For each $p\in R_{X(n, \delta)}$, there exists $U_p$ and $\{(a_{i},\,b_{i})\}_{i=1}^{n}$,  an $%
( n,\delta) $--strainer for $U_p$. By Part ($3$) of Lemma \ref{lemma3.9}, we may choose each $\phi_p$ to be orientation preserving.
Then $\left\{ \left( U_{p},\phi _{p}\right) \right\} _{p\in R_{X(n,\delta)}} $ is an orientation preserving atlas for $X$. 

Let $U_p \, \cap \, U_q \neq
\emptyset$ for  $p,\,q \in R_{X(n,\delta )}$, where $\{(a_{i},\,b_{i})\}_{i=1}^{n}$ and $\{(c_{i},\,d_{i})\}_{i=1}^{n}$ are the corresponding $(n,
\delta)-$strainers, respectively, for $U_p$ and $U_q$.
By Part ($2$) of Lemma \ref{lemma3.9}, It follows that 
 $$\det (\nabla (\phi _{p}\circ \phi _{q}^{-1}))>0,$$
wherever it is defined. Since $p$ and $q$ were arbitrarily chosen, the result follows.
\end{proof}


\section{The proof of Theorem A}\label{sec-t:1}

In this section, we assume that $(X,d)$ is an orientable, $n$-dimensional closed Alexandrov space with curvature bounded below by $\kappa$.
We will construct a $n$-dimensional integral current structure $T$ of weight $1$ on 
$(X,d)$ in two steps.  First, we define a $n$-dimensional integer rectifiable current structure $T$ on $(X,d)$.  Second, we prove that $(X,d,T)$ is indeed an $n$-integral current space, by showing that $\partial T=0$ and that $\set(T)=X$.  See Theorem \ref{equalcurr}
and Corollary \ref{setT=X}.


\subsection{Construction of an integer rectifiable current on ${\bf (X,d)}$}\label{ssec-t:1 step 1}

In this subsection, we complete the first step of the proof of Theorem A by proving Theorem \ref{Xinteger} below.
That is, we will construct a $n$-dimensional integer rectifiable current $T$ on $X$. 
To do so,  we utilize the existence of charts for points that are $(n, \epsilon)$-strained as described in Section\,\ref{sec-background}.

 \begin{theorem}\label{Xinteger} 
Let $(X, d)$ be a $n$-dimensional oriented Alexandrov space with curvature bounded below by $\kappa$ and diameter bounded above by $D$. Then there exists a $n$-dimensional integer rectifiable current  space $(X, d, T)$ with weight $1$, that is, $\theta_T=1$.  Moreover, $T$ is unique in the sense of Lemma \ref{unique}. 
 \end{theorem}

The proof of Theorem \ref{Xinteger} follows once we have proven the following four lemmas.

\begin{lemma}\label{lem-metricF}
 Let $(X, d)$ be a $n$-dimensional oriented Alexandrov space with curvature bounded below by $\kappa$. Then $T:\mathcal{D}^n(X) \to \R$ given by 
$$T = \sum_{i=1}^\infty \varphi_{i\#}\Lbrack \theta_i\Rbrack = \sum_{i=1}^\infty \varphi_{i\#}\Lbrack A_i\Rbrack$$
is a  multi-linear functional, where $\varphi_i: A_i \to X$ are bi-Lipchitz maps with mutually disjoint images and $A_i \subset \R^n$ are Borel sets, as constructed in Theorem \ref{atlas}.
\end{lemma}

\begin{proof}
 By Theorem \ref{atlas}, there exist a countable collection of bi-Lipschitz maps $\varphi_i: A_i \to X$ with $A_i \subset \R^n$ 
 precompact Borel measurable subsets. 
 
On each $A_i$, define the weight function $\theta_i: A_i \to \N$ by $\theta_i = 1_{A_i}$, where $1_{A_i}$ denotes the indicator function on $A_i$.  Set $\Lbrack A_i \Rbrack =  \Lbrack 1_{A_i} \Rbrack$ as in Example 2.17 of \cite{SW}, that is $\Lbrack A_i \Rbrack: \mathcal{D}^n(\R^n) \to \R$ 
$$\Lbrack A_i \Rbrack (f, \pi_1, \dots, \pi_n)= \int_{A_i \subset \R^n}  f \det\left(\nabla \pi_i\right) \, d\mathcal{L}^n.$$

Then define $T$ as follows:
$$T = \sum_{i=1}^\infty \varphi_{i\#}\Lbrack \theta_i\Rbrack = \sum_{i=1}^\infty \varphi_{i\#}\Lbrack A_i\Rbrack.$$

The first step consists in proving $T(f,\pi)\in\R$ for $(f,\pi)\in\mathcal{D}^n(X)$, that is, $T(f,\pi)$ is finite.\\
Note that for each $j\in\N$ we know that $T_j:=\varphi_{j\#}\Lbrack A_j\Rbrack: \mathcal{D}^n(\varphi_j(A_j))\rightarrow\R$ is a current. 
Thus, by ($3$) in Definition\,\ref{defn-current}, there exists a finite Borel measure $\mu_j$ such that 
\begin{align*}
\lvert T_j(f,\pi)\lvert\leq \prod_{i=1}^n\mbox{Lip}(\pi_i)\int_{\varphi_j(A_j)}\lvert f\lvert d\mu_j.
\end{align*}
The preceding inequality in particular holds true for the mass measure $\|T_j\|$ of $T_j$.\\
By the triangle inequality we thus get
\begin{align*}
\notag\lvert T(f,\pi)\lvert\leq\sum_{j=1}^{\infty}\lvert T_j(f,\pi)\lvert&\leq\sum_{j=1}^{\infty}(\prod_{i=1}^n\mbox{Lip}(\pi_i))\int_{\varphi_j(A_j)}\lvert f\lvert d\|T_j\|\\
\end{align*}
Hence
\begin{align}
\label{1}
\lvert T(f,\pi)\lvert\leq (\prod_{i=1}^n\mbox{Lip}(\pi_i))\|f\|_{\infty}\sum_{j=1}^{\infty}\|T_j\|({\varphi_j(A_j)}).
\end{align}

\smallskip

We now proceed by proving that the right hand side of the above inequality is finite.
Note
\begin{align}
\label{2}
\|\varphi_{j\#}\Lbrack \theta_j \Rbrack \| \leq \Lip(\varphi_j)^n \varphi_{j\#} \|\Lbrack 1_j \Rbrack \| 
 = \Lip(\varphi_j)^n \varphi_{j\#} (\mathcal{L}^n)
 =\Lip(\varphi_j)^n \varphi_{j\#} (\mathcal{H}^n),
 \end{align}                                                             
and
$$\varphi_{j\#} (\mathcal{H}^n) (\varphi_j(A_j)) =  \mathcal{H}^n (A_j) .$$
Since $\varphi_j$ is bi-Lipschitz, we have
$$\mathcal{H}^n (A_j) \leq \Lip( \varphi_j^{-1})^n \mathcal{H}^n ( \varphi_j (A_j)) .$$
Consequently, by Equation\,(\ref{2}) we have
$$\|\varphi_{j\#}\Lbrack \theta_j \Rbrack \|  (\varphi_j(A_j))  
\leq 
\Lip(\varphi_j)^n \Lip( \varphi_j^{-1})^n \mathcal{H}^n ( \varphi_j (A_j)).$$
Finally, we get
\begin{align*}
\sum_{j=1}^{\infty}\|\varphi_{j\#}\Lbrack \theta_j \Rbrack \|  ( \varphi_j (A_j)) &\leq\sum_{j=1}^{\infty} \Lip(\varphi_j)^n \Lip( \varphi_j^{-1})^n \mathcal{H}^n ( \varphi_j (A_j)) \\
&\leq c(n)\sum_{j=1}^{\infty}    \mathcal{H}^n ( \varphi_j(A_j)) \\
&= c(n) \mathcal{H}^n (\cup  \varphi_j (A_j)) \\
&\leq  c(n)  \mathcal{H}^n (X)\\
&\leq C(n,\kappa, D),
\end{align*}
where $c(n)=\Lip(\varphi_j)^n \Lip( \varphi_j^{-1})^n$ is a constant depending only on $n$ by Theorem \ref{atlas} and $C(n,\kappa, D)$ is  a constant depending only on the dimension $n$, the curvature bound $\kappa$ and the diameter bound $D$ given by Bishop Volume Comparison Theorem \ref{Grombishop Ineq}. Combining the preceding inequality with (\ref{1}) we get that $T(f,\pi)$ is finite.

\smallskip

Finally, the multilinearity of $T$ follows from the fact that each summand, $\varphi_{i\#}\Lbrack A_i\Rbrack$, is multilinear. 
\end{proof}

We now prove that $T$ is indeed a current.

\begin{lemma}\label{lem-Tiscurrent}
The multilinear functional $T$ defined as in Lemma \ref{lem-metricF} is a current.
\end{lemma}

\begin{proof}
We must verify that Definition \ref{defn-current} is satisfied. 
Property (1) in Definition\,\ref{defn-current} follows easily.
Indeed, if there exists a neighborhood on which $\pi_i$ is constant, then since $\varphi_{i\#}\Lbrack A_i\Rbrack$ is a current we have $\varphi_{i\#}\Lbrack A_i\Rbrack (f,\pi)=0$ for all $i\in\N$. Hence, $T(f,\pi)=0$.

\smallskip

Next we show property (2) in Definition\,\ref{defn-current}.
We have to prove that $\lim_{i\rightarrow\infty}T(f,\pi_1^i,\cdots,\pi_n^i)=T(f,\pi_1,\cdots,\pi_n)$, whenever $\pi_{k}^i$ converges pointwise to $\pi_k$ in $X$, where $\mbox{Lip}(\pi_j^i)\leq C$ for some constant $C$.
By the above considerations we know that $\sum_{j=1}^{\infty}\lvert T_j(f,\pi_1^i,\cdots,\pi_n^i)\lvert<\infty$, consequently the sum $\sum_{j=1}^{\infty} T_j(f,\pi_1^i,\cdots,\pi_n^i)$ converges absolutely.
Since $T_j$ is a current, for each $j\in\N$ we have $\lim_{i\rightarrow\infty}T_j(f,\pi_1^i,\cdots,\pi_n^i)=T_j(f,\pi_1,\cdots,\pi_n)$. We thus can commute the infinite sum and the limit and hence obtain the claimed equality
\begin{align*} 
 \lim_{i\rightarrow\infty}T(f,\pi_1^i,\cdots,\pi_n^i)&=\lim_{i\rightarrow\infty}\sum_{j=1}^{\infty}T_j(f,\pi_1^i,\cdots,\pi_n^i)\\&=\sum_{j=1}^{\infty}\lim_{i\rightarrow\infty}T_j(f,\pi_1^i,\cdots,\pi_n^i)
 \\&=\sum_{j=1}^{\infty}T_j(f,\pi_1,\cdots,\pi_n)=T(f,\pi_1,\cdots,\pi_n).
 \end{align*}

\smallskip

Finally, we show property (3) in Definition\,\ref{defn-current}.
By the above considerations we have 
\begin{align*}
\lvert T(f,\pi)\lvert\leq \sum_{j=1}^{\infty}\lvert T_j(f,\pi)\lvert\leq \prod_{i=1}^n\mbox{Lip}(\pi_i)\int_{\varphi_j(A_j)}\lvert f\lvert d\mu_j.
\end{align*}
We define $\mu=\sum_{j=1}^{\infty}\mu_j$, which defines a finite Borel measure on $X$ since the sets $\varphi_j(A_j)$ are disjoint. Thus we get
\begin{align*}
\lvert T(f,\pi)\lvert\leq \sum_{j=1}^{\infty}\lvert T_j(f,\pi)\lvert\leq \prod_{i=1}^n\mbox{Lip}(\pi_i)\int_{X}\lvert f\lvert d\mu,
\end{align*}
and hence (3) is established. Thus $T$ is a current.
\end{proof}

\begin{lemma}\label{lemTInteger} 
The current $T$ as defined in Lemma \ref{lem-Tiscurrent} is an integer rectifiable current.
\end{lemma}

\begin{proof}
We have to check that $T$ satisfies the conditions of Definition \ref{def-integercur}. By Lemma \ref{lem-Tiscurrent}, it suffices to show that $\mass(T)=\sum_{j=1}^{\infty}\mass(T_j)$. Since we have $\|T\|(\varphi_j(A_j))= \|T_j\|(\varphi_j(A_j))$ for all $j\in\N$, we get
$$\mass(T)=\|T\|(\cup\varphi_j(A_j))=\sum_{j=1}^{\infty}\|T\|(\varphi_j(A_j))=\sum_{j=1}^{\infty}\|T_j\|(\varphi_j(A_j))=\sum_{j=1}^{\infty}\mass(T_j),$$
where we used that the sets $\varphi_j(A_j)$ are disjoint.  It follows $\mass(T)=\sum_{j=1}^{\infty}\mass(T_j)$.
\end{proof}

\begin{lemma}\label{unique} 
Suppose that $T$ and $T'$ are two integer rectifiable currents defined on $(X,d)$ with current parametrizations $(\{\varphi_i\}, \{1_{A_i}\})$ and $(\{\varphi_i'\}, \{1_{A_i'}\})$
constructed using the $(n, \delta)$-strained points as described in Theorem \ref{atlas}.
Assume that these two atlases belong to the same maximal oriented atlas of $R_{X(n,\delta)}$.  Then $T=T'$.
\end{lemma}

\begin{proof}
We apply Lemma \ref{lem-equalcurr} to prove that $T$ equals $T'$.

\smallskip

Since $R_{X(n, \delta)}\subset \cup \varphi_i(A_i) \subset X$ and $R_{X(n, \delta)}\subset \cup \varphi_i'(A_i')\subset X$, it follows that the symmetric difference between $\cup \varphi_i(A_i)$ and $\cup\varphi_i'(A_i')$ is actually contained in $S_X$. Therefore, by Theorem \ref{Hausdim}, the symmetric difference has zero $n$-dimensional Hausdorff measure.

\smallskip

Furthermore, since $\{(A_i, \varphi_i)\}$ and $\{   (A'_i, \varphi'_i)\}$ are contained in the same maximal oriented atlas, it follows that 
\begin{align*}
\det(\nabla(\varphi^{-1}_i \circ \varphi'_j)) > 0 \text{   and    } 
\det(\nabla(\varphi'^{-1}_i \circ \varphi_j)) > 0.
\end{align*}
Finally, 
$$\theta_T(x)= \sum_{i=1}^\infty \theta_i\circ\varphi_i^{-1}(x)\One_{\varphi_i\left(A_i\right)}(x)
=1$$
and 
$$\theta_{T'}(x)= \sum_{i=1}^\infty \theta_i'\circ\varphi_i'^{-1}(x)\One_{\varphi_i'\left(A_i'\right)}(x)=1$$
for all $x \in R_X(n, \delta)$.  Thus by Lemma \ref{lem-equalcurr}, $T= T'$. 
\end{proof}


\subsection{${\bf (X,d,T)}$ is an integral current space.}
In this subsection we prove that $(X,d,T)$ is a $n$-dimensional integral current space. In order to do this, we first prove that $\partial T=0$. We accomplish this by showing that the integral current $T'$ defined on $X$, for which $\partial T'=0$ as in Mitsuishi \cite{M2} is, in fact, equal to the integer rectifiable current we define in Theorem \ref{equalcurr} below. We then prove that  $\set(T)=X$ in Corollary \ref{setT=X}. With these two steps, we have completed the proof of Theorem A.
 

\begin{theorem}\label{equalcurr}
Let $(X,d)$ be an $n$-dimensional closed Alexandrov space, $T$ the $n$-current defined in Theorem \ref{Xinteger} and $T'$ the $n$-current that generates the group $\{ S \in \intcurr_n\left(X\right) | \, \bdry S=0\} = \mathbb Z$ from Theorem \ref{thm-Mitsuishi}. Then, either $T=T'$ or $T=-T'$. Hence, $\bdry T=0$.
\end{theorem}

\begin{proof}
We will prove the theorem using Lemma \ref{lem-equalcurr}.
Recall, that by definition we have
$$T = \sum_{i=1}^\infty \varphi_{i\#}\Lbrack A_i\Rbrack\quad\mbox{and}\quad T'= \sum_{i=1}^N f_{i\#}\Lbrack \Delta^n\Rbrack.$$
We first prove Condition $1$ of Lemma \ref{lem-equalcurr}.
By the definitions of $T$ and $T'$, $R_{X(n, \delta)} \subset \bigcup_{i=1}^\infty \varphi_i(A_i)$ and $X=\bigcup_{i=1}^N f_i(\Delta^n)$, respectively. Then we have
\begin{equation*}
\begin{array}{ll}
\cup_{i=1}^N f_i(\Delta^n)  \,  \mathbf{\Delta}  \cup_{i=1}^\infty \varphi_i(A_i) &=  \cup_{i=1}^N f_i(\Delta^n) \setminus \cup_{i=1}^\infty \varphi_i(A_i)  \bigcup \cup_{i=1}^\infty \varphi_i(A_i)\setminus  \cup_{i=1}^N f_i(\Delta^n) \\
& \subset  X \setminus R_{X(n, \delta)}.\\
\end{array}
\end{equation*}
By Theorem \ref{Hausdim},  $\mathcal H^n (X \setminus R_{X(n, \delta)})=0$. Hence,  $$\mathcal H^n (\cup_{i=1}^N f_i(\Delta^n)  \,  \mathbf{\Delta}  \cup_{i=1}^\infty \varphi_i(A_i))=0,$$
which establishes the first item of Lemma \ref{lem-equalcurr}.

Next we prove Condition $2$ of Lemma \ref{lem-equalcurr}. We must show that $T$ and $T'$ have the same orientation, that is, 
in all overlapping sets we have
\begin{equation*}
\det(\nabla(\varphi^{-1}_i \circ f_j)) > 0 \text{   and    } 
\det(\nabla(f^{-1}_i \circ \varphi_j)) > 0.
\end{equation*}
The orientation on $T'$ is given by a generator of $H_n(X; \mathbb{Z})$ and 
by Lemma \ref{strained-orientable} this provides us with an orientation on $R_{X(n, \delta)}=\cup_{i=1}^\infty \phi_i(A_i) \cap \sum_{i=1}^N f_{i\#}\Lbrack \Delta^n\Rbrack$. It then follows by Lemma \ref{strained-orientable-grad} that it defines an orientation
on both $(X, d, T)$ and $(X, d, T')$ as desired. 

Finally,  we prove Condition $3$ of Lemma \ref{lem-equalcurr}. By definition of $T$, $\theta_T=1$ on  $\cup_{i=1}^\infty \varphi_i(A_i)$ which has full measure in $X$. Now, since $\{ S \in \intcurr_n\left(X\right) | \bdry S=0\} = H_n (X;\mathbb{Z})$,   the restriction $f_i|_{\overset{\circ}{\Delta^n}} $is injective. Then, $\theta_{T'}=1$ on $\bigcup_{i=1}^N f_i(\overset{\circ}{\Delta^n})= X \setminus \bigcup_{i=1}^N f_i(\partial \Delta^n)$.  Since $\partial \Delta^n$ has $(n-1)$- Hausdorff dimension, then its $n$-Hausdorff measure equals zero. Since the $f_i$ are Lipschitz maps, then $\mathcal H^n ( \bigcup_{i=1}^N f_i(\partial \Delta^n))=0$. Hence, $\theta_{T'}$ equals 1 almost everywhere in X.  Hence, $\theta_T=\theta_{T'}$ almost everywhere. 

Thus,  $T=T'$, and the result follows. 
\end{proof}

\smallskip

In the next lemma we show that the regular points of $X$ are contained in $\set(T)$. 
The subsequent corollary shows $\set(T)=X$.

\begin{lemma}\label{lem-regInSet} Let $(X, d)$ be a $n$-dimensional Alexandrov space with curvature bounded below by $\kappa$. Let  $T$ be an integer current structure on $X$ defined as in Proposition \ref{Xinteger}. Let $p\in R_{X(n, \delta)}$, then the following hold.
\begin{enumerate}
\item The point $p$ is contained in $\set(T)$. That is, the $n$-dimensional lower density of $\|T\|$ at $p$ is positive.
\item The density of $\|T\|$ at a regular point $p$ is equal to $1$.
\end{enumerate}
\end{lemma}

\begin{proof}
From the definition of $\set(T)$ we have to prove that 
\begin{align*}
\liminf_{r\to 0} \frac{\|T\|(B_{r} (p))}{\omega_n r^n}> 0.
\end{align*}
\smallskip

Since $p$ is a $(n,\delta)$-strained point of $X$,  Theorem \ref{bilip}  states that there is a neighborhood $U_p$ of $p$ and a bi-Lipschitz map 
$$f\colon U_p \to   W \subset  \R^n$$ 
such that 
\begin{equation}\label{LIP}
 \Lip(f)^{-1} d(x,y) \leq d(f(x),f(y)) \leq \Lip(f) d(x,y), 
 \end{equation}
for all $x,y\in  U_p$.

\smallskip

Given $\varphi: Y \to Z$ a Lipschitz map and $S$ an $n$-dimensional current defined on $Y$
it follows from Inequality (2.4) in \cite{AK} that
\begin{equation}\label{2.4}
\varphi_\# \| S \| \geq (\Lip(\varphi))^{-n}  \| \varphi_\# S \|.
\end{equation}

Let $r_0\in\R$ be such that $B_{r_0}(p)\subset U_p$.
Applying Inequality \ref{2.4}  to the function $f$ and $T$, for any $r \leq r_0$ we have 
$$ \| T \| (B_r(p)) = f_\# \| T \| (f (B_r(p))) \geq
 \Lip(f)^{-n}    \| f_\# T \| (f (B_r(p))).$$ Using Inequality \ref{LIP}, we obtain  
$$
 B_{r\Lip(f)^{-1}}(f(p)) \subset f( B_{r}(p)).$$

 Then $$  \| f_\# T \| (f (B_{r}(p)))  \geq    \| f_\# T \| (B_{r\Lip(f)^{-1}}(f(p))) = 
 \omega_n r^n \Lip(f)^{-n},$$ where the equality comes from $U_p \subset \bigcup_{i=1}^\infty \varphi_i (A_i)$ which gives $\|f_\# T\|  = \mathcal L^n$ in $f(U_p)$. Putting together the last two inequalities we obtain
\begin{equation}\label{eq-massballLow}
\| T \| (B_r(p))  \geq  \omega_n \Lip(f)^{-2n} r^n.
\end{equation}
We conclude that 
$$\liminf_{r\to 0} \frac{\|T\|(B_{r} (p))}{\omega_n r^n} \geq \Lip(f)^ {-2n} > 0.$$
Thus $p \in \set(T)$ and so Part (1) is established. 

\medskip

We now prove Part (2).
Using once again Inequality \ref{LIP}, we obtain 
$$ f( B_{r}(p))\subset B_{r\Lip(f)}(f(p)).$$
We apply (\ref{2.4}) with $\varphi=f^{-1}$ and $S=f_\# T$. Since $(f^{-1})_\# f_\# T = T$ we thus get
\begin{align}\label{eq-massballUp}
\| T \| (B_r(p))& \leq \Lip(f)^{-n} \| f_\# T \| (f(B_r(p))) \leq \Lip(f)^{-n} \| f_\# T \| (B_{r\Lip(f)} (f(p)))\\\notag& =  \Lip(f)^{-n} \omega_n r^n \Lip(f)^{n} = \omega_n r^n.
\end{align}

\smallskip
Since $p$ is a regular point,  $\Lip(f)$ can be made arbitrarily close to $1$ in Inequality (\ref{eq-massballLow}). 
Then, by Inequalities (\ref{eq-massballLow}) and (\ref{eq-massballUp}) we get,
$$\lim_{r\to 0} \frac{\|T\|(B_{r} (p))}{\omega_n r^n} = 1.$$
Hence the density of $\|T\|$ at $p$ is equal to $1$.
\end{proof}

\begin{corollary}\label{setT=X}
Let $(X,d)$ be an $n$-dimensional Alexandrov space with curvature bounded below by $\kappa$. Let $T$ be the integer current previously defined on $X$. Then $X=\set(T)$.
\end{corollary}

\begin{proof}
This proof follows almost exactly the last part of the proofs given in Theorem 7.1 \cite{SW10} and Theorem 0.1 \cite{LP}, where is shown that under certain conditions the Gromov-Hausdorff limit and the intrinsic flat limit of a sequence agree.

\smallskip

By Lemma \ref{lemma-weight}, for $x \in X$ we have
\begin{equation}
\|T\|(B_r(x))= \int_{B_r(x)}\theta_T(y)\lambda(y)d\mathcal H^n\rstr \set (T),
\end{equation}
where $\theta_T: \overline{\set(T)} \to \N \cup \{0\}$ is an integrable function with
$\theta_T > 0$ in $\set(T)$ and $\lambda: \set(T) \to \R$ is a non-negative integrable function that satisfies $\lambda \geq n^{-n/2}$. Using this last inequality and $\theta_T=1$ we obtain
\begin{equation}\label{eq-massballD}
\|T\|(B_r(x)) \geq n^{-n/2} \mathcal{H}^n(B_r(x) \cap \set(T)).
\end{equation}

Furthermore, by Lemma \ref{lem-regInSet} we know $R_{X}\subset \set(T)\subset X$.
Therefore, $\mathcal H^n (\set(T) \setminus R_{X})=0$ and thus we obtain
\begin{equation}\label{eq-haussball}
\mathcal{H}^n(B_r(x) \cap \set(T)) =\mathcal{H}^n(B_r(x) \cap R_{X})=\mathcal{H}^n(B_r(x)).
\end{equation}
Then by Equation (\ref{eq-massballD}) and Equation (\ref{eq-haussball}) we get
\begin{equation}\label{eq-dens}
\frac{\|T\|(B_r(x))}{\omega_nr^n}\geq\frac{n^{-n/2} \mathcal H^n(B_r(x))}{\omega_nr^n}.
\end{equation}

\smallskip

For $\kappa \geq 0$, using the Bishop-Gromov volume comparison theorem for Alexandrov spaces, Theorem \ref{Grombishop Ineq}, we have for $R>r>0$:
$$
 \mathcal H^n(B_r(x)) \geq r^n \mathcal H^n(X) / R^n.
$$
It follows that $\liminf_{r\to 0} \|T\|(B_r(x))/\omega_nr^n  \geq  \frac{n^{-n/2}   r^n \mathcal H^n(X)}{\omega_nr^n R^n}   >0$.

\smallskip

For $\kappa < 0$, using once more the Bishop-Gromov volume comparison theorem for Alexandrov spaces, Theorem \ref{Grombishop Ineq}, with $R>r>0$, we have:
$$\mathcal{H}^n(B_r(x))\geq \frac{\mathcal H ^n(B_R(x))}{V_{\kappa}(R)} V_{\kappa}(r),$$
where $V_{\kappa}(r)$ and  $V_{\kappa}(R)$ denote respectively the volumes of the $r$-ball  and the $R$-ball in the space form $M_{\kappa}^n$. Thus, from inequality\,(\ref{eq-dens}) we get
$$\frac{\|T\|(B_r(x))}{\omega_nr^n}\geq n^{-n/2} \frac{\mathcal H ^n(B_R(x))}{V_{\kappa}(R)} \frac{V_{\kappa}(r)}{\omega_nr^n}.$$
Now, $V_{\kappa}(r)\geq V_{0}(r)= \omega_n r^n$ by Bishop's Inequality (see Theorem \ref{bishop Ineq}). It follows that $$\liminf_{r\to 0} \|T\|(B_r(x))/\omega_nr^n\geq n^{-n/2} \frac{\mathcal H ^n(B_R(x))}{V_{\kappa}(R)} > 0.$$ 
This shows that $X \subset \set(T)$.
\end{proof}

\begin{remark} The current we have constructed is canonical in the sense that if the Alexandrov space, $X$, consists only of regular points, that is, if it is a smooth manifold, then $\lambda=1$ and the current $T(\omega)= \int_X \omega$ has weight equal to $1$. Hence, $\| T\|= \mathcal H^n= \mathcal L^n =\vol_X$. 
\end{remark}

\end{document}